\documentclass[11pt]{article}
\usepackage[margin=1in]{geometry} 
\usepackage{amssymb,amsfonts,amsmath,bbm,mathrsfs,stmaryrd}
\usepackage{xcolor}
\usepackage{url}

\usepackage{enumerate}

\usepackage[colorlinks,
             linkcolor=black!75!red,
             citecolor=blue,
             pdftitle={},
             pdfauthor={Alexandru Chirvasitu, Michael Brannan},
             pdfproducer={pdfLaTeX},
             pdfpagemode=None,
             bookmarksopen=true
             bookmarksnumbered=true]{hyperref}

\usepackage{tikz}
\usetikzlibrary{arrows,calc,decorations.pathreplacing,decorations.markings,intersections,shapes.geometric,through,fit,shapes.symbols,positioning,decorations.pathmorphing}

\usepackage{braket}

\usepackage[amsmath,thmmarks,hyperref]{ntheorem}
\usepackage{cleveref}

\creflabelformat{enumi}{#2(#1)#3}

\crefname{section}{Section}{Sections}
\crefformat{section}{#2Section~#1#3} 
\Crefformat{section}{#2Section~#1#3} 

\crefname{subsection}{\S}{\S\S}
\crefformat{subsection}{#2\S#1#3} 
\Crefformat{subsection}{#2\S#1#3} 

\theoremstyle{plain}

\newtheorem{lemma}{Lemma}[section]
\newtheorem{proposition}[lemma]{Proposition}
\newtheorem{corollary}[lemma]{Corollary}
\newtheorem{theorem}[lemma]{Theorem}

\theoremstyle{nonumberplain}

\theoremstyle{plain}
\theorembodyfont{\upshape}
\theoremsymbol{\ensuremath{\blacklozenge}}

\newtheorem{definition}[lemma]{Definition}

\newtheorem{remark}[lemma]{Remark}

\crefname{definition}{definition}{definitions}
\crefformat{definition}{#2definition~#1#3} 
\Crefformat{definition}{#2Definition~#1#3} 

\crefname{ex}{example}{examples}
\crefformat{example}{#2example~#1#3} 
\Crefformat{example}{#2Example~#1#3} 

\crefname{remark}{remark}{remarks}
\crefformat{remark}{#2remark~#1#3} 
\Crefformat{remark}{#2Remark~#1#3} 

\crefname{convention}{convention}{conventions}
\crefformat{convention}{#2convention~#1#3} 
\Crefformat{convention}{#2Convention~#1#3}

\crefname{lemma}{lemma}{lemmas}
\crefformat{lemma}{#2lemma~#1#3} 
\Crefformat{lemma}{#2Lemma~#1#3} 

\crefname{proposition}{proposition}{propositions}
\crefformat{proposition}{#2proposition~#1#3} 
\Crefformat{proposition}{#2Proposition~#1#3} 

\crefname{corollary}{corollary}{corollaries}
\crefformat{corollary}{#2corollary~#1#3} 
\Crefformat{corollary}{#2Corollary~#1#3} 

\crefname{theorem}{theorem}{theorems}
\crefformat{theorem}{#2theorem~#1#3} 
\Crefformat{theorem}{#2Theorem~#1#3} 

\crefname{assumption}{assumption}{Assumptions}
\crefformat{assumption}{#2assumption~#1#3} 
\Crefformat{assumption}{#2Assumption~#1#3} 

\crefname{equation}{}{}
\crefformat{equation}{(#2#1#3)} 
\Crefformat{equation}{(#2#1#3)}

\theoremstyle{nonumberplain}
\theoremsymbol{\ensuremath{\blacksquare}}

\newtheorem{proof}{Proof}
\newcommand\pf[1]{\newtheorem{#1}{Proof of \Cref{#1}}}

\newcommand\bC{{\mathbb C}}

\newcommand\bN{{\mathbb N}}

\newcommand\bZ{{\mathbb Z}}

\newcommand\cO{{\mathcal O}}

\DeclareMathOperator{\id}{id}

\newcommand\numberthis{\addtocounter{equation}{1}\tag{\theequation}}

\title{Hopf algebras with enough quotients}
\author{Alexandru Chirvasitu}

\begin{document}

\date{}

\newcommand{\Addresses}{{% additional braces for segregating \footnotesize
  \bigskip
  \footnotesize

  \textsc{Department of Mathematics, University at Buffalo, Buffalo,
    NY 14260-2900, USA}\par\nopagebreak \textit{E-mail address}:
  \texttt{achirvas@buffalo.edu}

% %   \medskip
% %   
% %   \textsc{Department of Mathematics, institution,
% %     address}\par\nopagebreak \textit{E-mail address}:
% %   \texttt{??}
% % 
}}

\maketitle

\begin{abstract}
  A family of algebra maps $H\to A_i$ whose common domain is a Hopf algebra is said to be jointly inner faithful if it does not factor simultaneously through a proper Hopf quotient of $H$. We show that tensor and free products of jointly inner faithful maps of Hopf algebras are again jointly inner faithful, generalizing a number of results in the literature on torus generation of compact quantum groups.
\end{abstract}


\vspace{.5cm}


%\tableofcontents

%%%%%%%%%%%%%%%%%%%%%%%%%%%%%%%%%%%%%%%%%%%%%%%%%%%%%%%%%%%%%%%%%%%%%%%%%%%%%%%%%%%%%%%%%%%%%%%%%%%%%%%%%%%%%%%%%%
%%%%%%%%%%%%%%%%%%%%%%%%%%%%%%%%%%%%%%%%%%%%%%%%%%%%%%%%%%%%%%%%%%%%%%%%%%%%%%%%%%%%%%%%%%%%%%%%%%%%%%%%%%%%%%%%%%
\section*{Introduction}

Inner faithfulness has been a recurring theme in the quantum group literature:

\begin{definition}\label{def.if}
  An algebra morphism $H\to A$ from a Hopf algebra to an algebra is {\it inner faithful} if it does not factor through any proper Hopf quotients of $H$. 
\end{definition}

The intuition is that if $H$ is thought of as a ``quantum'' version of a group algebra $k\Gamma$ the group representation $k\Gamma\to A$ factors through no proper quotient groups of $\Gamma$, i.e. is faithful in the group-theoretic sense.

The concept plays a central role in \cite{bb-inner}, where the related notion of an {\it inner linear} Hopf algebra was introduced: a Hopf algebra having a finite-dimensional inner faithful representation (i.e. morphism to some matrix algebra $M_n$). More generally, the same paper introduces the {\it Hopf image} of a morphism $H\to A$, meaning the largest Hopf quotient of $H$ factoring said morphism.

Other sources where inner linearity/faithfulness appear in various guises include \cite{ban-subf,ban-mor,ban-nico,bbs} in the context of compact quantum groups, \cite{kv-vn,vae-out} in the broader setting of locally compact quantum groups, \cite{ab} for plain Hopf algebras, etc.

In the same circle of ideas, the notion of {\it topological generation} for compact quantum groups has seen some recent attention. Introduced in \cite{bcv} and perused extensively e.g. in \cite{ban-hl,ban-tor,ban-ax,ban-urfl} and (occasionally under different terminology) in \cite{chi-rfd,bw,bbcw,bcf} it is essentially a multi-map generalization of \Cref{def.if}:

A quantum group dual to a Hopf algebra $H$ is said to be topologically generated by the quantum subgroups corresponding to Hopf algebra morphisms $\pi_i:H\to A_i$ if the latter morphisms do not factor simultaneously through a proper Hopf quotient of $H$. Focusing on the Hopf algebras rather than their attached ``quantum groups'', we refer to this property of a family of (algebra or Hopf algebra) morphisms $H\to A_i$ as {\it joint inner faithfulness} (see \Cref{def.jnt} below).

The present paper is motivated by the interest in ``permanence'' results for joint inner faithfulness under various natural operations, such as tensor and free products. \cite[Proposition 4.5]{ban-tor} and \cite[Theorem 3.4]{chi-gen} are examples of such, and are recovered and generalized by the main results of this paper (see \Cref{pr.tens,pr.free} for more precise formulations):

\begin{theorem}
  A tensor or free product of two jointly inner faithful families of Hopf algebra morphisms is again jointly inner faithful.
\end{theorem}

%%%%%%%%%%%%%%%

The short \Cref{se.prel} mostly sets up some terminology and conventions while \Cref{se.main} contains the main results of the paper. Finally, in \Cref{se.other} explains how various results in the recent literature are recoverable from the present paper.

%%%%%%%%%%%%%%%%%%%%%%%%%%%%%%%%%%%%%%%%%%%%%%%%%%%%%%%%%%%%%%%%%%%%%%%%%%%%%%%%%%%%%%%%%%%%%%%%%%%%%%%%%%%%%%%%%%
\subsection*{Acknowledgements}

This work is partially supported by NSF grant DMS-1801011.

%%%%%%%%%%%%%%%%%%%%%%%%%%%%%%%%%%%%%%%%%%%%%%%%%%%%%%%%%%%%%%%%%%%%%%%%%%%%%%%%%%%%%%%%%%%%%%%%%%%%%%%%%%%%%%%%%%
%%%%%%%%%%%%%%%%%%%%%%%%%%%%%%%%%%%%%%%%%%%%%%%%%%%%%%%%%%%%%%%%%%%%%%%%%%%%%%%%%%%%%%%%%%%%%%%%%%%%%%%%%%%%%%%%%%
\section{Preliminaries}\label{se.prel}

We work with Hopf algebras over a fixed but arbitrary field $k$ that will be implicit throughout. Although we focus on plain Hopf algebras for simplicity, the discussion goes through virtually unchanged for related categories, such as Hopf $*$-algebras over the complex numbers or the CQG algebras of \cite{dk}.

\begin{definition}\label{def.jnt}
  Let $H$ be a Hopf algebra and $\pi_i:H\to A_i$ a family of algebra morphisms.

  The family $\{\pi_i\}$ is {\it jointly inner faithful} (or jointly IF, or just IF) if the $\pi_i$ factor through no proper Hopf quotient $H\to \overline{H}$.

  $\{\pi\}$ is {\it jointly faithful} (or simply faithful) is the $\pi_i$ factor through no proper {\it algebra} quotient $H\to \overline{H}$.
\end{definition}

In some of the results below it will make a difference whether the maps $\pi_i:H\to A_i$ are plain algebra morphisms or Hopf algebra morphisms; we emphasize the distinction where appropriate. 

$H^{op}$ (and $H^{cop}$) denote $H$ equipped, respectively, with its opposite (co)multiplication. To merge the two opposite structures we use the symbol $H^{op,cop}$.

%%%%%%%%%%%%%%%%%%%%%%%%%%%%%%%%%%%%%%%%%%%%%%%%%%%%%%%%%%%%%%%%%%%%%%%%%%%%%%%%%%%%%%%%%%%%%%%%%%%%%%%%%%%%%%%%%%
%%%%%%%%%%%%%%%%%%%%%%%%%%%%%%%%%%%%%%%%%%%%%%%%%%%%%%%%%%%%%%%%%%%%%%%%%%%%%%%%%%%%%%%%%%%%%%%%%%%%%%%%%%%%%%%%%%
\section{Main results}\label{se.main}

The first result deals with tensor (rather than free) products.

\begin{theorem}\label{pr.tens}
  If
  \begin{equation*}
    \pi_i:H\to A_i,\ i\in I\quad\text{and}\quad \pi_j:K\to B_j,\ j\in J
  \end{equation*}
  are jointly IF families of Hopf algebra morphisms then so is
  \begin{equation*}
    \pi_i\otimes \pi_j:H\otimes K\to A_i\otimes B_j,\ (i,j)\in I\times J.
  \end{equation*}
\end{theorem}

Next, we consider free products.

\begin{theorem}\label{pr.free}
  If
  \begin{equation*}
    \pi_i:H\to A_i,\ i\in I\quad\text{and}\quad \pi_j:K\to B_j,\ j\in J
  \end{equation*}
  are jointly IF families of Hopf algebra morphisms then so is
  \begin{equation*}
    \pi_i* \pi_j:H* K\to A_i* B_j,\ (i,j)\in I\times J.
  \end{equation*}
\end{theorem}

This will require some auxiliary results.

\begin{proposition}\label{pr.z2}
    If
  \begin{equation*}
    \pi_i:H\to A_i,\ i\in I
  \end{equation*}
  is a jointly IF family of Hopf algebra morphisms then so is
  \begin{equation}\label{eq:1}
    \eta_i:=\id* \pi_i:k\bZ_2 * H\to k\bZ_2* A_i.
  \end{equation}
\end{proposition}

\begin{remark}\label{re.not-alg}
  Note that in \Cref{pr.tens}, \Cref{pr.free} and so on we are considering families of morphisms of {\it Hopf} algebras (rather than just algebras).
\end{remark}

Before going into any of the proofs we need some notation. For an index set $I$ denote by $L_I$ (standing for the {\it language} attached to $I$) the set of words ${\bf i}$ in letters $i\in I$ and {\it adjoint copies} $i^*$ for $i\in I$. We equip $L_I$ with the involution `$*$' interchanging $i\in I$ and their adjoint copies $i^*$ and reversing juxtaposition. For such a word ${\bf i}\in L_I$ we write $|{\bf i}|$ for its {\it length} (i.e. number of letters).

Now, for ${\bf i}=i_1\cdots i_k\in L_I$ denote
\begin{equation}\label{eq:9}
  A_{\bf i}:= A_{i_1}\otimes \cdots\otimes A_{i_k}
\end{equation}
where
\begin{equation*}
  A_{i}=
  \begin{cases}
    A_{i}&\text{if }i\in I\numberthis\label{eq:ai-cases}\\
    A_{i^*}^{op}&\text{if }i^*\in I.
  \end{cases}
\end{equation*}

For $i\in I$ we write
\begin{equation*}
 \begin{tikzpicture}[auto,baseline=(current  bounding  box.center)]
  \path[anchor=base] (0,0) node (1) {$H$} +(2,.5) node (2) {$H^{op,cop}$} +(5,.5) node (3) {$A_i^{op}$} +(7,0) node (4) {$A_{i^*}$};
  \draw[->] (1) to[bend left=6] node[pos=.5,auto] {$\scriptstyle S$} (2);
  \draw[->] (2) to[bend left=6] node[pos=.5,auto] {$\scriptstyle \pi_i^{op}$} (3);
  \draw[->] (3) to[bend left=6] node[pos=.5,auto] {$\scriptstyle =$} (4);
  \draw[->] (1) to[bend right=6] node[pos=.5,auto,swap] {$\scriptstyle \pi_{i^*}$} (4);
 \end{tikzpicture}
\end{equation*}
where $S$ is the antipode.

Finally, for a word
\begin{equation}\label{eq:2}
 {\bf i}=i_1\cdots i_k\in L_I 
\end{equation}
define the morphism
\begin{equation}\label{eq:3}
  \pi_{\bf i}:H\to A_{\bf i}
\end{equation}
by
\begin{equation*}
  \begin{tikzpicture}[auto,baseline=(current  bounding  box.center)]
    \path[anchor=base] (0,0) node (1) {$H$} +(3,.5) node (2) {$H^{\otimes |{\bf i}|}$} +(8,0) node (3) {$A_{\bf i}$};
    \draw[->] (1) to[bend left=6] node[pos=.5,auto] {$\scriptstyle \Delta^{(|{\bf i}|)}$} (2);
    \draw[->] (2) to[bend left=6] node[pos=.3,auto] {$\scriptstyle \pi_{i_1}\otimes\cdots \otimes \pi_{i_k}$} (3);
    \draw[->] (1) to[bend right=6] node[pos=.5,auto,swap] {$\scriptstyle \pi_{\bf i}$} (3);
 \end{tikzpicture}
\end{equation*}
where
\begin{equation*}
  \Delta^{(n)}:H\to H^{\otimes n}
\end{equation*}
denotes the obvious iterated comultiplication. The same construction applies to the maps $\eta_i$ of \Cref{eq:1} to give morphisms
\begin{equation*}
  \eta_{\bf i}:k\bZ_2*H \to(k\bZ_2*A_{i_1})\otimes\cdots\otimes (k\bZ_2*A_{i_k}) 
\end{equation*}
for \Cref{eq:2}. When having to refer back explicitly to the index set $I$ (as will be the case below in the proof of \Cref{pr.tens}) we write $\pi^I_{\bf i}$ for $\pi_{\bf i}$.

As a final piece of notation, we apply the notation \Cref{eq:ai-cases} to $H$ itself: 
\begin{equation*}
  H_{i}=
  \begin{cases}
    H_{i}&\text{if }i\in I\numberthis\label{eq:h-cases}\\
    H_{i^*}^{op}&\text{if }i^*\in I.
  \end{cases}
\end{equation*}
This also allows us to write $H_{\bf i}$ as in \Cref{eq:9} for
\begin{equation*}
  {\bf i}=i_1\cdots i_k\in L_I
\end{equation*}
and
\begin{equation}\label{eq:10}
  \Delta^{({\bf i})}:H\to H_{\bf i}
\end{equation}
for $\Delta^{(|{\bf i}|)}$ composed with antipodes applied to those tensorands in
\begin{equation*}
  H_{\bf i} = H_{i_1}\otimes \cdots \otimes H_{i_k}
\end{equation*}
corresponding to $i_s\in I^*$.

\begin{remark}\label{re.bb-strings}
  Similar formalism is introduced in \cite[\S 2.1]{bb-inner}, which considers single maps $H\to A$ (i.e. $I=\{\alpha\}$ would be a singleton). In that setup our $L_I$ can be identified with the free involutive monoid generated by the single element $\alpha$ (with multiplication-reversing involution defined by $\alpha\mapsto \alpha^*$). Our $A_{\bf i}$ are instead parametrized in \cite{bb-inner} by the free monoid $F$ generated by $\bN$ (with involution interchanging the even and odd non-negative integers): the parametrization has some redundancy, with $F$ mapping surjectively (but not injectively) onto $L_I$ by sending even elements to $\alpha$ and odd ones to $\alpha^*$.
\end{remark}

\begin{remark}\label{re.joint}
  The joint IF property means precisely that as ${\bf i}$ ranges over $L_I$, the maps $\pi_{\bf i}$ in \Cref{eq:3} are jointly faithful: their kernels intersect to zero.

  For Hopf algebra morphism families $\pi:H\to A_i$ we can say more however. Consider a morphism $\pi_{\bf i}$ as in \Cref{eq:3} and a word ${\bf j}$ formed by inserting any letter $j$ in $I\cup I^*$ anywhere in ${\bf i}$:
  \begin{equation*}
    {\bf i} = {\bf i}_l{\bf i}_r\text{ and }{\bf j} = {\bf i}_l\ j\ {\bf i}_r. 
  \end{equation*}
Composing $\pi_{\bf j}$ with
\begin{equation*}
  \id^{\otimes |{\bf i}_l|}\otimes \varepsilon_{A_j} \otimes  \id^{\otimes |{\bf i}_r|} : A_{\bf j}\to A_{\bf i}
\end{equation*}
produces $\pi_{\bf i}$, so the kernel of the latter must contain $\ker \pi_{\bf j}$. 
  
% %   Note furthermore that both
% %   \begin{equation*}
% %     \begin{tikzpicture}[auto,baseline=(current  bounding  box.center)]
% %       \path[anchor=base] (0,0) node (1) {$A_i$} +(2,0) node (2) {$A_i\otimes A_i$};
% %       \draw[->] (1) to[bend left=0] node[pos=.5,auto] {$\scriptstyle \Delta $} (2);
% %     \end{tikzpicture}
% %   \end{equation*}
% % and
% %   \begin{equation*}
% %     \begin{tikzpicture}[auto,baseline=(current  bounding  box.center)]
% %       \path[anchor=base] (0,0) node (1) {$A_i$} +(2,0) node (2) {$A_i\otimes A_i$} +(5,0) node (3) {$A_{i^*}\otimes A_i$};
% %       \draw[->] (1) to[bend left=0] node[pos=.5,auto] {$\scriptstyle \Delta $} (2);
% %       \draw[->] (2) to[bend left=0] node[pos=.5,auto] {$\scriptstyle S\otimes\id $} (3);
% %     \end{tikzpicture}
% %   \end{equation*}
% %   are one-to-one (e.g. because composing them further with $\varepsilon\otimes\id$ produces the identity).
% % 
% %   We can thus always compose $\pi_{\bf i}$ with one of these two maps (applied to one of the tensorands of the codomain $A_{\bf i}$ of $\pi_{\bf i}$) to obtain $\pi_{\bf i'}$ such that
% %   \begin{itemize}
% %   \item ${\bf i'}$ is obtained from ${\bf i}$ by adding a single character anywhere;
% %   \item the additional character is in either $I$ or $I^*$, as preferred (i.e. we can choose either);
% %   \item the kernel of $\pi_{\bf i'}$ is contained in that of $\pi_{\bf i}$. 
% %   \end{itemize}
% % 
% %   

This provides us with \Cref{le.order} below. Moreover, we apply this same observation to enlarge some of the words ${\bf i}$ occurring in the proof of \Cref{pr.tens} below.

Note that for this argument to go through, it is crucial to be working with Hopf algebra maps rather than plain algebra morphisms.
\end{remark}

For the purpose of stating \Cref{le.order} we regard $L_I$ as a poset with order `$\le$' defined by declaring ${\bf j}\le {\bf i}$ if ${\bf j}$ can be obtained from ${\bf i}$ by eliminating some letters (and leaving the leftovers in the same order). 

\begin{lemma}\label{le.order}
  Let $\pi_i:H\to A_i$ be a family of Hopf algebra maps. Then, the correspondence
  \begin{equation*}
    L_I\ni {\bf i}\mapsto \ker \pi_{\bf i}
  \end{equation*}
  is order-reversing, where the lattice of subspaces of $H$ is ordered by inclusion and $L_I$ as in the discussion preceding the statement. 
\end{lemma}
\begin{proof}
  This follows from \Cref{re.joint}, which argues that inserting letters arbitrarily into a word ${\bf i}$ can only shrink $\ker \pi_{\bf i}$.
\end{proof}

\begin{lemma}\label{le.filter}
  Let $\pi_i:H\to A_i$ be a jointly IF family of Hopf algebra maps. Then, for every finite linearly independent set $F\subset H$, there is some ${\bf i}\in L_I$ such that
  \begin{equation*}
    \pi_{\bf i}(F)\subset A_{\bf i}
  \end{equation*}
  is linearly independent. 
\end{lemma}
\begin{proof}
  Denote by $V$ the span of $F$. We know that the intersection
  \begin{equation}\label{eq:6}
    \bigcap_{\bf i}(V\cap \ker \pi_{\bf i})
  \end{equation}
  is trivial. Moreover, we know from \Cref{le.order} that said intersection is {\it filtered}: every finite family of spaces $V\cap \ker\pi_{{\bf i}_t}$ contains another such space $V\cap \ker \pi_{\bf i}$ for some word ${\bf i}$ that dominates all ${\bf i}_t$ in the order we have equipped $L_I$ with. Since all mentioned spaces are finite-dimensional, the vanishing of their filtered intersection \Cref{eq:6} implies that one of them must be trivial.
\end{proof}

\pf{pr.tens}
\begin{pr.tens}
  Consider a non-zero element $x\in H\otimes K$. It is expandable as a linear combination
  \begin{equation*}
    x=\sum_{\alpha,\beta} c_{\alpha\beta}f_{\alpha}\otimes g_{\beta}
  \end{equation*}
  of simple tensors $f_\alpha\otimes g_\beta$ for finite linearly independent set $\{f_{\alpha}\}\subset H$ and $\{g_{\beta}\}\subset K$.

  By \Cref{le.filter} there are tuples ${\bf i}\in L_I$ and ${\bf j}\in L_J$ such that the images of $\{f_{\alpha}\}$ and $\{g_{\beta}\}$ through $\pi^I_{\bf i}$ and $\pi^J_{\bf j}$ respectively are linearly independent, and hence $x$ is not annihilated by the map
  \begin{equation}\label{eq:4}
    \pi^I_{\bf i}\otimes \pi^J_{\bf j}:H\otimes K\to A_{\bf i}\otimes B_{\bf j}. 
  \end{equation}

  Enlarging one of ${\bf i}$ and ${\bf j}$ if necessary via the procedure described in \Cref{re.joint}, we can assume that they have the same number of components:
  \begin{equation*}
    {\bf i}=i_1\cdots i_k\text{ and }{\bf j}=j_1\cdots j_k. 
  \end{equation*}
  Furthermore, by enlarging both words as appropriate we can ensure that for each index $1\le u\le k$ we have either
  \begin{equation*}
    (i_u,j_u)\in I\times J\text{ or } (i_u,j_u)\in I^*\times J^*.
  \end{equation*}
We thus have
  \begin{equation}\label{eq:7}
    {\bf ij}:=(i_1,j_1)\cdots(i_k,j_k)\in L_{I\times J}
  \end{equation}
  and in this case the map \Cref{eq:4} is nothing but
  \begin{equation}\label{eq:8}
    \pi_{\bf ij}^{I\times J}:H\otimes K\to (A_{i_1}\otimes B_{j_1})\otimes\cdots\otimes (A_{i_k}\otimes B_{j_k}).
  \end{equation}
  Since it fails to annihilate the arbitrary element $0\ne x\in H\otimes K$, the conclusion follows. 
\end{pr.tens}

% % 
% %   Writing
% %   \begin{equation*}
% %     {\bf i}\times {\bf j}:=\{(i,j)\ |\ i\in {\bf i},\ j\in {\bf j}\},
% %   \end{equation*}
% %   the map
% %   \begin{equation}\label{eq:5}
% %     \pi^{I\times J}_{{\bf i}\times {\bf j}}:H\otimes K\to (H\otimes K)_{{\bf i}\times {\bf j}} = \bigotimes_{i,j}A_i\otimes B_j
% %   \end{equation}
% %   factors as \Cref{eq:4} followed by an embedding
% %   \begin{equation*}
% %     A_{\bf i}\otimes B_{\bf j}\subseteq (H\otimes K)_{{\bf i}\times {\bf j}}. 
% %   \end{equation*}
% %   It follows that $x$ is not annihilated by \Cref{eq:5}, finishing the proof.
% % 

\pf{pr.z2}
\begin{pr.z2}
  Let $x\in k\bZ_2 * H$ be a non-zero element and $t\in \bZ_2$ the generator of that group. Then, having fixed a basis $\{e_{\alpha}\}$ for $H$, $x$ can be written as a linear combination of words of the form
  \begin{equation*}
    te_{\alpha_1}te_{\alpha_2}\cdots
  \end{equation*}
  and analogues (i.e. the starting $t$ might be absent, the last letter might be a $t$ or an $e_{\alpha}$, etc.).

  Due to our assumption that the family $\{\pi_i\}$ is jointly IF and \Cref{le.filter} there is some ${\bf i}$ as in \Cref{eq:2} such that the images through $\pi_{\bf i}$ of the $e_{\alpha}$ appearing in the decomposition of $x$ are linearly independent. The morphism
  \begin{equation*}
    \eta_{\bf i}:k\bZ_2*H \to (k\bZ_2*A_{i_1})\otimes\cdots\otimes (k\bZ_2*A_{i_k})  
  \end{equation*}
  corresponding to ${\bf i}$ factors through the subalgebra $A$ of the codomain generated by
  \begin{equation*}
    A_{\bf i} = A_{i_1}\otimes\cdots\otimes A_{i_k}
  \end{equation*}
  and the diagonally-embedded
  \begin{equation*}
    \Delta^{(k)}:k\bZ_2 \to k\bZ_2^{\otimes k}\subset (k\bZ_2*A_{i_1})\otimes\cdots\otimes (k\bZ_2*A_{i_k}).
  \end{equation*}
  Since $A\cong k\bZ_2 * A_{\bf i}$, the choice of ${\bf i}$ (making the images of the $e_{\alpha}$ appearing in $x$ linearly independent) ensures that $\eta_{\bf i}$ does not annihilate $x$. The non-zero element $x\in k\bZ_2 * H$ being arbitrary, this finishes the proof.
\end{pr.z2}

\begin{corollary}\label{cor.hh}
    If
  \begin{equation*}
    \pi_i:H\to A_i,\ i\in I
  \end{equation*}
  is a jointly IF family of Hopf algebra morphisms then so is
  \begin{equation*}
    \pi_i* \pi_j:H * H\to A_i* A_j,\ i,j\in I.
  \end{equation*}  
\end{corollary}
\begin{proof}
  Let $\sigma\in \bZ_2$ be the generator. The subalgebra of $k\bZ_2 * H$ generated by $H$ and $\sigma H \sigma$ is isomorphic to $H*H$ upon identifying
  \begin{equation*}
    H\ni x\mapsto \sigma x\sigma\in \sigma H\sigma.
  \end{equation*}
  The conclusion now follows from \Cref{pr.z2} by restricting the jointly IF family $\id *\pi_i$ to
  \begin{equation*}
    H*H\cong H*\sigma H\sigma\subset k\bZ_2 * H, 
  \end{equation*}
  finishing the proof.
\end{proof}

\pf{pr.free}
\begin{pr.free}
  We will reduce the problem to the case $H=K$ (and identical jointly IF families) as follows.

  Embed $H$ and $K$ into $H\otimes K$ in the obvious fashion, and restrict the single (jointly IF, by \Cref{pr.tens}) family
  \begin{equation*}
    \pi_i\otimes \pi_j:H\otimes K\to A_i\otimes B_j
  \end{equation*}
  to $H$ and $K$ to recover $\pi_i$ and $\pi_j$ respectively.

  Since we now have an embedding
  \begin{equation*}
    H * K \subset (H\otimes K) * (H\otimes K)
  \end{equation*}
  compatible with the $\pi$ map families in the sense that $(\pi_i\otimes \pi_j)^{*2}$ restricts to $\pi_i * \pi_j$, in order to prove the joint IF property for the latter it suffices to do so for the former. But this, in turn, follows from the joint IF-ness of $\pi_i\otimes \pi_j$ and \Cref{cor.hh} applied to this the Hopf algebra $H\otimes K$ and this single family $(\pi_i\otimes \pi_j)_{i,j}$.
\end{pr.free}

%%%%%%%%%%%%%%%%%%%%%%%%%%%%%%%%%%%%%%%%%%%%%%%%%%%%%%%%%%%%%%%%%%%%%%%%%%%%%
%%%%%%%%%%%%%%%%%%%%%%%%%%%%%%%%%%%%%%%%%%%%%%%%%%%%%%%%%%%%%%%%%%%%%%%%%%%%%
\section{Other results in the literature}\label{se.other}

The preceding material sheds some light on recent results in the same spirit. In \cite{ban-tor}, for instance, the main problem being discussed is whether compact quantum groups are generated by their tori (see e.g. \cite[Conjecture 2.3]{ban-tor} as well as \cite[Introduction]{ban-pat} or \cite{bbd,bv-inv}):

\begin{definition}\label{def.gen-tor}
  Let $H$ be a Hopf algebra, regarded as an algebra of functions on a quantum group $G$. We say that $G$ is {\it generated by its tori} if the family of surjections from $H$ onto its quotient group Hopf algebras $H\to k\Gamma$ is jointly IF. 
\end{definition}

In the context of ``CQG algebras'' (i.e. cocommutative Hopf algebras with positive unital integral; \cite{dk}) \cite[Proposition 4.5]{ban-tor} reads

\begin{proposition}\label{pr.gen-tor}
If the quantum groups attached to the Hopf algebras $H$ and $K$ are generated by their tori then so is the quantum group associated to the free product $H*K$.   
\end{proposition}
\begin{proof}
This is an immediate consequence of \Cref{pr.free}.   
\end{proof}

In the same spirit, \cite[Theorem 3.4]{chi-gen} shows that for a compact connected Lie group $G$ with Hopf algebra of representative functions $\cO(G)$ the family of morphisms
\begin{equation*}
  \bC[t,t^{-1}]*\cO(G)\to \bC[t,t^{-1}]*\cO(T)
\end{equation*}
is jointly IF for $T$ ranging over the maximal tori $T$. It too is thus a consequence of \Cref{pr.free}. 

As noted before (e.g. in \Cref{re.joint}) working with families of Hopf algebra maps was crucial in much of the discussion above. We can, however, recover some of the results in weakened form.

As an example, consider \cite[Proposition 7.4]{bb-inner} as our starting point. It deals with single maps (i.e. the families are singletons) and shows that the tensor product of two IF algebra morphisms is again IF provided one of the maps is actually faithful and the respective Hopf algebra has injective antipode. We have an analogous result for families:

\begin{proposition}
  Let
  \begin{equation*}
    \pi_i:H\to A_i,\ i\in I\quad\text{and}\quad \pi_j:K\to B_j,\ j\in J
  \end{equation*}
  be families of algebra morphisms with ${\pi_i}$ jointly faithful, $\{\pi_j\}$ jointly IF, and $H$ having injective antipode. Then, $\{\pi_i\otimes\pi_j\}_{i,j}$ is jointly IF. 
\end{proposition}
\begin{proof}
  Note that since tensor products of algebras commute with finite products in each tensorand we may as well assume that $\{\pi_i\}$ is closed under finite products, i.e. every finite product
  \begin{equation*}
    \prod_{i}\pi_i : H\to \prod_i A_i
  \end{equation*}
is again in our family. Together with the joint faithfulness this will ensure that for every finite linearly independent set $F\subset H$ there is some $i\in I$ such that $\pi_i(F)\subset A_i$ is linearly independent. 

Let $x\in H\otimes K$ be a non-zero element and expand it as
\begin{equation*}
  x=\sum e_s\otimes f_s
\end{equation*}
where
\begin{itemize}
\item $e_s\in H$ are finitely many linearly independent elements;
\item $f_s\in K$ are non-zero. 
\end{itemize}
Because $\{\pi_j\}_{j}$ is jointly IF there is some ${\bf j}\in L_J$ such that $\pi_{\bf j}(f_s)\ne 0$ in $B_{\bf j}$ for at least one $s$. Writing
\begin{equation*}
  {\bf j}=j_1\cdots j_k,\ j_t\in J\cup J^*,
\end{equation*}
we will be done once we argue that there is some
\begin{equation*}
 {\bf i}=i_1\cdots i_k \in L_I 
\end{equation*}
such that 
\begin{enumerate}[(a)]
\item\label{item:1} $i_t\in I$ if and only if $j_t\in J$ (and similarly for $I^*$ and $J^*$ respectively);
\item $\pi_{\bf i}(e_s)\in A_{\bf i}$ are linearly independent;
\end{enumerate}
indeed, in that case $\pi^{I\times J}_{\bf ij}$ defined as in \Cref{eq:7,eq:8} will fail to annihilate the arbitrary non-zero element $x\in H\otimes K$.

To see that ${\bf i}\in L_I$ with the requisite properties exists note first that all maps $\Delta^{({\bf i})}$ as in \Cref{eq:10} are injective by our assumption that the antipode of $H$ is. Now simply choose $i_t$, $1\le t\le k$ so as to ensure that for every $t$, $\pi_{i_t}:H\to A_{i_t}$ is one-to-one on the finite-dimensional span of the position-$t$ tensorands of the elements
\begin{equation*}
  \Delta^{({\bf i})}(e_s)\in H_{\bf i}. 
\end{equation*}
First off this makes sense because condition \labelcref{item:1} above means that $H_{\bf i}$ and $\Delta^{({\bf i})}$ are uniquely defined by ${\bf j}$, regardless of the specific choice of $i_t$. The fact that the $i_t$ can then be chosen as described in the preceding paragraph then follows from our assumptions of joint faithfulness and closure of $\{\pi_i\}_i$ under products.
\end{proof}

As noted above, this in particular recovers the case of singleton families of maps treated in \cite[Proposition 7.4]{bb-inner}.

% % It is also naturally applicable to {\it residually finite-dimensional} Hopf algebras, which are the main focus of \cite{chi-rfd} for instance:
% % 
% % \begin{definition}
% %   An algebra is {\it residually finite-dimensional} if its family of 
% % \end{definition}
% % 

%%%%%%%%%%%%%%%%%%%%%%%%%%%%%%%%%%%%%%%%%%%%%%%%%%%%%%%%%%%%%%%%%%%%%%%%%%%%%%%%%%%%%%%%%%%%%%%%%%%%%%%%%%%%%%%%%%
%%%%%%%%%%%%%%%%%%%%%%%%%%%%%%%%%%%%%%%%%%%%%%%%%%%%%%%%%%%%%%%%%%%%%%%%%%%%%%%%%%%%%%%%%%%%%%%%%%%%%%%%%%%%%%%%%%

%\bibliography{uogn.bib}{}
\bibliographystyle{plain}
\addcontentsline{toc}{section}{References}

\def\polhk#1{\setbox0=\hbox{#1}{\ooalign{\hidewidth
  \lower1.5ex\hbox{`}\hidewidth\crcr\unhbox0}}}

\Addresses

\end{document}